\documentclass{scrartcl}
\usepackage{amsthm}
\usepackage{cite}    
\usepackage{graphicx}
\usepackage{psfrag}
\usepackage{subfigure} 
\usepackage{url}       
\usepackage{amsmath}   
\usepackage{algpseudocode}
\usepackage{enumerate}
\usepackage{amssymb}
\usepackage{units}

\usepackage[english]{babel}

\addto\captionsenglish{}

\newtheorem{theorem}{Theorem}

\DeclareMathOperator{\real}{Re}
\DeclareMathOperator{\imag}{Im}

\let\bf\textbf
\let\tt\texttt

\begin{document}

\title{\LARGE \bf Globally Optimal AC Power System Upgrade Planning under
  Operational Policy Constraints}

\author{ 
  Sandro Merkli\thanks{Automatic Control Lab,
    ETH Zurich, Physikstrasse 3, 8092 Zurich
    \tt{ \{smerkli,rsmith\}@control.ee.ethz.ch}}
  \thanks{
    embotech GmbH, Technoparkstrasse 1, 8005 Zurich,
    \tt{\{merkli,jerez,domahidi\}@embotech.com}},
  Alexander Domahidi\footnotemark[2], 
  Juan Jerez\footnotemark[2],
  Roy S.\ Smith\footnotemark[1] \\
}
\date{May 2018\\[0.5cm]\small{Copyright Notice: The final version of this work
  has been presented at the European Control Conference 2018 in Limassol, Cyprus.
  The EUCA holds copyright on that version.}}

\def \todo {\textbf{todo:} }
\def \div { \mathrm{div} }
\def \rot { \mathrm{\textbf{rot}} }
\def \minfty { -\infty }
\def \infint { \int_{-\infty}^\infty }
\def \infsum { \sum_{k = -\infty}^\infty }
\def \adj { \mathrm{adj} }
\def \lps { \quad \laplace \quad }
\def \B { \mathbb B }
\def \C { \mathbb C }
\def \N { \mathbb N }
\def \R { \mathbb R }
\def \Z { \mathbb Z }
\def \Q { \mathbb Q }
\def \Ac { \mathcal A }
\def \Bc { \mathcal B }
\def \Dc { \mathcal D }
\def \Ec { \mathcal E }
\def \Fc { \mathcal F }
\def \Gc { \mathcal G }
\def \Xc { \mathcal X }
\def \Ic { \mathcal I }
\def \Zc { \mathcal Z }
\def \Qc { \mathcal Q }
\def \Rc { \mathcal R }
\def \Uc { \mathcal U }
\def \Nc { \mathcal N }
\def \ei { \varepsilon_{\text{int}} }
\def \xtt { x_{t|t}   }
\def \gdw { \;\; \Longleftrightarrow \;\; }
\def \diag { \operatorname{diag} }
\def \minim { \operatorname*{minimize} }
\def \maxim { \operatorname*{maximize} }
\def \st { \operatorname*{subject\ to} }
\def \dom { \operatorname{dom} }
\def \img { \operatorname{im} }
\def \conv { \operatorname{conv} }
\def \cone { \operatorname{cone} }
\def \projop { \operatorname{proj} }
\def \jw  { j\omega }
\def \bmb { \begin{bmatrix} }
\def \bme { \end{bmatrix} }
\newcommand{\cve}[1] {#1}
\newcommand{\jcve}[1] {\bar{#1}}
\newcommand{\drop}[1] {{ }}

\newtheorem{lemma}{Lemma}

\maketitle
\thispagestyle{empty}
\pagestyle{empty}
\begin{abstract}

In order to accommodate the increasing amounts of renewable generation in power
  distribution systems, system operators are facing the problem of how to
  upgrade transmission capacities.

Since line and transformer upgrades are costly, optimization procedures are
used to find the minimal number of upgrades required. 

The resulting design optimization formulations are generally mixed-integer
non-convex problems. 

Traditional approaches to solving them are usually of a heuristic nature,
yielding no bounds on suboptimality or even termination. 

In contrast, this work combines heuristics, lower-bounding procedures and
practical operational policy constraints. 

The resulting algorithm finds both suboptimal solutions quickly and the global
solution deterministically by a Branch-and-Bound procedure augmented with
lazy cuts.

\end{abstract}
\section{Introduction}
Increased penetration of renewables in distribution systems is leading to
voltage violations at peak generation times~\cite{Walling2008}. Both changes to
the operational schemes of the systems~\cite{Merkli2017} as well as changes to
the system topologies themselves~\cite{Rider2007} are being investigated as
remedies to these problems. This work treats the particular problem of
selecting which lines of a power grid to upgrade in order to stop voltage
violations from recurring in the future. Since significant investments are
required for infrastructure upgrades, optimization for minimum cost is desired.

Power system planning optimizations based on the AC model of the system
generally lead to non-convex optimization problems due to the non-linearity of
the equations involved. Additionally, integrality is introduced into the
planning problems from two sources: Firstly, the number of lines to be upgraded
is usually the dominant factor in the cost. This leads to a cardinality-type
cost that can be formulated using binary variables. Secondly, there are often
only a given number of different line admittance values available for purchase.
The presence of this integrality can be modeled with integer decision
variables. Due to the integrality and non-convexity of the problem, most
existing methods for solving it are heuristic in nature. One line of research
applies stochastic algorithms to the problem~\cite{Ramirez-Rosado1998}. Another
approach is to approximate the system model by the DC power flow
equations~\cite{Macedo2016}, turning the mixed-integer nonlinear program
(MINLP) into a mixed-integer linear program (MILP). Mature codes for solving
the latter are commercially available. However, the solution of the MILP
approximation is not necessarily feasible for the original MINLP. A third
approach is to form a convex outer approximation of the non-convex AC model
based on conic constraints, leading to a mixed-integer semidefinite problem
(MISDP)~\cite{Jabr2013}. Compared to the DC power flow approximation, this
approach has the advantage of yielding a guaranteed lower bound on the
optimal cost of the planning problem.

The method presented in this paper extends the third approach based on MISDP in
two aspects: Firstly, additional constraints are added to the problem that
specify how the system will be operated given fixed upgrade decisions and load
patterns. These constraints will hereafter be referred to as ``policy
constraints''. The addition of policy constraints guarantees that the solution
of the planning problem is feasible when deployed in practice. Additionally it
circumvents the difficulty of checking feasibility of a power flow for a given
load pattern, which was recently shown to be NP-hard~\cite{Bienstock2015,Lehmann2016}.
Secondly, the Branch-and-Bound algorithm is adapted to incorporate the policy
constraints efficiently. The result is a method that deterministically finds
the globally optimal solution for an optimization problem that better reflects
reality than the commonly used MINLP without policy constraints. Due to the
nature of the Branch-and-Bound procedure, any existing heuristics for finding
useful feasible points as well as existing lower-bounding methods can easily be
integrated into the approach. 
Despite the problem still being NP-hard to solve, the numerical experiments
presented towards the end of the paper demonstrate that the method can be
applied successfully to practical systems.

\subsection{Outline}
Section~\ref{sec:prelim} introduces notation, the model, policy constraints and
violating snapshots. It ends with a formulation of the upgrade problem solved
in this work. Section~\ref{sec:bnbalgo} introduces the Branch-and-Bound
algorithm used and the implementation of policy constraints as lazy cuts.
Section~\ref{sec:numsim} presents numerical simulations with different
operating policies and discusses properties of the results. 

\section{Modeling}
\label{sec:prelim}
\subsection{Notation and Kirchhoff equations}
The complex conjugate of a complex variable $x$ is denoted by $\bar x$. We will
use tuple notation $(a,b)$ where appropriate to refer to a collection of
vectors $a$ and $b$ to avoid introducing overly many stacked vectors.
The power grid is modeled using an undirected graph with $N$ vertices
(representing buses) and $L$ edges (representing lines). Each line (say,
between buses $j$ and $l$) has an admittance $y_{jl} \in \C$. Associated with
each bus $j$ is a voltage $v_j \in \C$ and a power in-feed $s_j \in \C$. The
power flows in the grid are governed by the AC Kirchhoff equations:
\begin{equation}
  \label{eqn:kirchhoff}
  \diag(v) \bar Y \bar v = s,
\end{equation}
where $Y \in \C^{N \times N}$ is the system admittance matrix:
\begin{equation}
  Y_{jl} := \begin{cases} y_{jl} & \text{if } j \ne l, \\
    y_{j}^{\mathrm{sh}} - \sum_{k=1,k\ne j}^N y_{jk} & \text{if } j = l,
  \end{cases}
\end{equation}
with $y_j^{\mathrm{sh}} \in \C$ being shunt admittances.

\subsection{Operation-independent constraints}
Certain constraints are needed no matter how the system is operated.
Firstly, the voltage magnitude at each bus $j$ is constrained to an admissible
range:
\begin{equation}
  \label{eqn:vdonut}
  v_{\min,j} \le |v_j| \le v_{\max,j},
\end{equation}
and the currents over the lines (for example, line $(l,j)$) are constrained to
be within thermal limits:
\begin{equation}
  \label{eqn:linecurr}
  |Y_{jl}||v_j - v_l | \le I_{\max,jl}.
\end{equation}
Assume the power at each bus $j$ has upper and lower bounds: 
\begin{equation}
  \label{eqn:powbox}
  s_{\min,j} \le s_j \le s_{\max,j},
\end{equation}
which can be interpreted as generator capability bounds for generator buses.
At load buses, the bounds can either represent the capability for demand
response or they can be fixed to the specific load drawn at that bus in 
a given scenario. For ease of notation, define the set of voltages and 
powers satisfying these constraints as 
\begin{equation}
  \mathcal X := \left\{ (s,v)\; \big|\; \eqref{eqn:vdonut}, \eqref{eqn:linecurr}, 
    \eqref{eqn:powbox} \right\}.
\end{equation}
In the case of multiple distinct system states being considered, we will make
use of the notation $\mathcal X^k$ for the constraints applying for system
state $k$. Note that the set $\mathcal X$ is parametrized by $Y$ and
$I_{\max,jl}$.

Additional constraints can be included in the approach provided they admit
semidefinite relaxations and can be included in the policy. For example, one
way of including phase angle constraints is in the form of 
\begin{equation}
  \label{eqn:langlecon}
  \left| v_j - v_l \right|^2 \le \sin(\alpha)^2 \left| v_j \right|,
\end{equation}
for a line from bus $j$ to bus $l$ and an angle limit of $\alpha$.
Constraint~\eqref{eqn:langlecon} can be added to the problem in quadratic form
and handled similarly to the voltage magnitude constraints.

\subsection{Operational policy}
Assume now that the buses of the network are partitioned into loads and
generators. Accordingly, let $s_\mathrm{gen}$ be the vector of powers of the
generators and $s_\mathrm{load}$ be that of the loads. 
Given a system topology and load pattern, we assume the operator of
the generators picks set-points (and with them, the network voltages) based on
a given operating policy:
\begin{equation} 
  (s_\mathrm{gen},v) = g(Y,s_\mathrm{load}).
\end{equation}
Note that we make the distinction between variables that are under operator
control (specifically $s_\mathrm{gen}$ and by extension also~$v$ here) and
variables the operator can only measure or estimate (specifically
$s_\mathrm{load}$ here). The method presented in this work does not require this
specific distinction, only that a distinction is made. One could for example 
also include curtailable loads in the set of variables under operator control.

The operating policy function $g$ can be any surjective function, the only
requirement is that it can be evaluated with reasonable efficiency. Examples
for operational policies include: 
\begin{itemize}
  \item[-] Local frequency-based controllers at the generators;
  \item[-] Networked frequency-based controllers;
  \item[-] Power flow computations with fixed generator voltages;
  \item[-] Numerical economic dispatch computations.
\end{itemize}
Note that in all these operational policies, the resulting system state $(s,v)$
has to satisfy the Kirchhoff equations~\eqref{eqn:kirchhoff} in order to be
physically meaningful. However, whether $(s,v)$ is in $\mathcal X$ or not
can depend on the policy. For example, it could happen that while local
controllers might not find a workable solution for a situation in which the
system is strongly loaded, numerical optimization methods can. 

\subsection{Line upgrades}
\label{ssec:lineupg}

Line upgrades are modeled as changes to the Laplacian matrix~$Y$ of the grid:
For each upgrade, a binary variable $a_i$ encodes whether the upgrade is
performed ($1$) or not ($0$). A constant matrix $\delta Y_i \in \C^{N \times
N}$ is used to encode the change this upgrade brings to the system. Overall,
the upgraded~$Y$ can then be written as follows:
\begin{equation}
  Y_\mathrm{upg} = Y + \sum_{i=1}^{n_u} (a_i \cdot \delta Y_i),
\end{equation}
where $n_u$ is the number of upgrade possibilities.
This formulation allows for multiple line upgrades to be linked into one
upgrade possibility as well as for shunt admittance changes. The upgrade
variables are now collected in a vector $a \in \{0,1\}^{n_u}$ and are assumed
to be the only influence on system topology in the problem.

The change in the line current limits introduced by the upgrades is similarly
defined as 
\begin{equation}
  I_{\max,jl,\mathrm{upg}} = I_{\max,jl} + \sum_{i \in \mathcal U_{jl}} 
    (a_i\cdot \delta I_{jl,i}),
\end{equation}
where $\mathcal U_{jl}$ is the set of upgrade indices that affect the line from
bus $j$ to bus $l$ and $\delta I_{jl,i}$ is the change in current limit
introduced to that line by the upgrade $i$.

\subsection{Violating snapshots}
We assume that a set of problematic steady-state snapshots~$(s^k,v^k)$, $k \in
\{1,\ldots,K\}$ are given. These snapshots are assumed to satisfy the Kirchhoff
equations~\eqref{eqn:kirchhoff}, but violate some of the voltage and current
constraints in~\eqref{eqn:vdonut} and~\eqref{eqn:linecurr}. 
Snapshot data can come from a variety of sources: Examples include past
measurement data, predictions of future scenarios or known worst case loadings.
It is useful to look at past violating snapshots due to the repetitiveness of
load patterns: If a violating load pattern is observed, it is likely that a
similar pattern will show up again in the future.
The assumption is made here that the uncertainty and variability of renewables
is represented adequately by the selection and number of snapshots. A more
precise representation of uncertainty is difficult to formulate at this time
due to the non-convexity of the upgrade problem and the presence of policy
constraints.

\subsection{System upgrade problem}
The problem we address in this work is to find upgrades to the power system
that would eliminate all constraint violations in the considered snapshots by
means of grid upgrades. The problem we would like to solve can be written as
follows:
\begin{subequations}
  \makeatletter
  \def\@currentlabel{U}
  \makeatother
  \label{eqn:upgradeproblem}
  \renewcommand{\theequation}{U.\arabic{equation}}
  \begin{align}
    \text{Problem U:} &\;\; \minim_{a,\tilde v^k,\tilde s^k} \;\; f(a) \\
    \st &\;\; Aa \le b \\
        &\;\; a \in \{0,1\}^{n_u} \label{eqn:up_upg}\\ 
        &\;\; Y_{\mathrm{upg}} = Y + \sum_{i=1}^{n_u} 
          (a_i \cdot \delta Y_i) \label{eqn:up_upg2} \\
        &\;\; I_{\max,jl,\mathrm{upg}} = I_{\max,jl} + 
          \sum_{i \in \mathcal U_{jl}} (a_i\cdot \delta I_{jl,i}), \\
        &\;\; \diag(\tilde v^k)\overline{Y_{\mathrm{upg}}
          \tilde v^k} = \tilde s^k \label{eqn:up_kir}\\
        &\;\; (\tilde s^k,\tilde v^k) \in \mathcal{X}^k \label{eqn:up_opcon}\\
        &\;\; (s_\mathrm{gen}^k,v^k) = g(Y_{\mathrm{upg}},s_\mathrm{load}^k) \\
        &\;\; \forall k \in \{1,\ldots, K\} \nonumber,
  \end{align}
\end{subequations}
where $f(a)$ represents the cost of the upgrades and is assumed here to be
convex. Convexity of $f(a)$ is required for the branch and bound procedure to
work. The polyhedron $Aa \le b$ represents constraints on upgrade combinations
-- for example that only one line type can be chosen per location. Note also
that given a value for $a$, the value of $Y_\mathrm{upg}$ can be computed
using~\eqref{eqn:up_upg2}. Therefore, the policy can also be understood as
being a function of $a$ directly and we will use the more compact notation
$g(a,s_\mathrm{load})$ from now on. The variable vectors $\tilde s^k$ and
$\tilde v^k$ are defined as follows:
\begin{equation}
  \tilde s^k_i = \begin{cases} 
    s^k_i & \text{if bus $i$ is a load} \\
    \text{a decision variable} & \text{otherwise.} \\
  \end{cases}
\end{equation}
The voltage vector $\tilde v^k$ is entirely an optimization variable: We expect all
voltages in the network to change if the topology and dispatch change. For
PV buses, the magnitude can be fixed by setting the lower and upper limits
appropriately. While it may seem like the presence of the
constraint~\eqref{eqn:up_opcon} renders the Kirchhoff
constraint~\eqref{eqn:up_kir} irrelevant, the latter will become important
later on for relaxation purposes.

\section{Branch-and-Bound procedure including policy constraints}
\label{sec:bnbalgo}
\subsection{QCQP formulation and semidefinite relaxation}
For the sake of easier notation, we assume there is only one snapshot (i.e.,
$K = 1$) and omit the $k$ index in this section. The resulting reformulation
then applies to each snapshot $k$ separately.
As shown in~\cite{Jabr2013,Low2014b}, among others, it is possible to apply a
set of reformulations to the intersection of the
constraints~\eqref{eqn:up_upg2} through~\eqref{eqn:up_opcon}. The result of
these reformulations is a set of quadratic constraints
\begin{equation}
  \alpha_i \le z^TQ_iz + q_i^Ty + m_i^Ta \le \beta_i
\end{equation}
for $i \in \{1,\ldots,I\}$ and where $z$ contains the voltage variables ($z =
\begin{bmatrix} \real{v}^T & \imag{v}^T \end{bmatrix}^T$), $a$ contains the 
upgrade variables as before and $y$ contains all the remaining variables.
The latter include the generator powers as well as additional variables that
were introduced in the reformulation. A detailed outline of the reformulation
used in this work is omitted here for brevity but can be found
in~\cite{Merkli2017a}. This means problem~\eqref{eqn:upgradeproblem} can be
equivalently rewritten as 
\begin{subequations}
  \makeatletter
  \def\@currentlabel{P}
  \makeatother
  \label{eqn:upgradeproblem2}
  \renewcommand{\theequation}{P.\arabic{equation}}
  \begin{align}
    \text{Problem P:} & \;\; \minim_{a,z,y} f(a) \\
         \st &\;\; Aa \le b \label{eqn:up2_poly}\\
             &\;\; a \in \{0,1\}^{n_u}\label{eqn:up2_bin} \\ 
             &\;\; \alpha_i \le z^TQ_iz + q_i^Ty + m_i^Ta \le \beta_i\label{eqn:up2_quad} \\
             &\;\; (s_\mathrm{gen},v) = g(a,s_\mathrm{load})\label{eqn:oppol} \\
             &\;\; \forall i \in \{1,\ldots, I\}.  \nonumber
  \end{align}
\end{subequations}
A mixed-integer convex relaxation of~\eqref{eqn:upgradeproblem2} can be found
by applying the commonly used semidefinite relaxation procedure for quadratic
constraints, which is outlined in~\cite{Low2014a}, as well as omitting the
policy constraint. This leads to a mixed-integer semidefinite problem:
\begin{subequations}
  \makeatletter
  \def\@currentlabel{R}
  \makeatother
  \label{eqn:upgradeproblem_cvx}
  \renewcommand{\theequation}{R.\arabic{equation}}
  \begin{align}
    \text{Problem R:} & \;\; \minim_{a,Z,y} f(a) \\
         \st &\;\; Aa \le b \\
             &\;\; a \in \{0,1\}^{n_u} \label{eqn:up3_bin}\\ 
             &\;\; \alpha_i \le tr(Q_i^TZ) + q_i^Ty + m_i^Ta \le \beta_i \label{eqn:up3_lmi}\\
             &\;\; Z \succeq 0 \label{eqn:up3_psd} \\
             &\;\; \forall i \in \{1,\ldots, I\} \nonumber,
  \end{align}
\end{subequations}
where $Z$ represents $zz^T$, but relaxed to $Z \succeq 0$ (details on this
relaxation are in~\cite{Low2014a}, among others).
Problem~\eqref{eqn:upgradeproblem_cvx} can now be solved to global optimality
using the Branch-and-Bound method, as demonstrated in~\cite{Jabr2013}.

\subsection{Policy cuts}
While computable, the solution of~\eqref{eqn:upgradeproblem_cvx} (denoted by
$(a^r,Z^r,y^r)$) is not guaranteed to be feasible
for~\eqref{eqn:upgradeproblem2} due to the former being a relaxation of the
latter.
Additionally, even if a $z^r$ can be found such that $Z^r = (z^r)(z^r)^T$
and $(a^r,z^r,y^r)$ satisfies constraints~\eqref{eqn:up2_poly}
through~\eqref{eqn:up2_quad}, there is no guarantee that the (likely different)
operating point that results when the operating policy
$g(a^r,s^r_{\mathrm{load}})$ is applied still satisfies them.  Assume that we are
given a solution $(a^r,z^r,y^r)$ of~\eqref{eqn:upgradeproblem_cvx} which is not
feasible for~\eqref{eqn:upgradeproblem2}. 
A constraint (or ``cut'') can then be added to~\eqref{eqn:upgradeproblem_cvx}
that excludes this particular upgrade choice $a^r$: 
\begin{equation}
  \label{eqn:lazycut}
  \|a - a^r\|_1 \ge 1, 
\end{equation}
The cut~\eqref{eqn:lazycut} only removes the particular case $a = a^r$ from the
feasible space of~\eqref{eqn:upgradeproblem_cvx} due to the entries of $a$
being in $\{0,1\}$. Such constraints will hereafter be referred to as ``policy
cuts''. They can be implemented as linear constraints on $a$. For notational
convenience, define the projection of the feasible set
of~\eqref{eqn:upgradeproblem2} onto the upgrade space as follows:
\begin{equation}
  \mathcal F_{\eqref{eqn:upgradeproblem2}} := \left\{ a \in \{0,1\}^{n_u}
    \; \middle| \;
    \begin{matrix} Aa \le b, \\
      \exists (z,y,v,s_\mathrm{gen}) : \eqref{eqn:up2_quad}, \eqref{eqn:oppol} \\
      \forall i \in \{1,\ldots,I\}
    \end{matrix} 
  \right\}.
\end{equation}
In other words, $\mathcal F_{\eqref{eqn:upgradeproblem2}}$ is the set of all
upgrade configurations $a$ for which~\eqref{eqn:upgradeproblem2} is feasible.
Similarly, for~\eqref{eqn:upgradeproblem_cvx}, we define
\begin{equation}
  \mathcal F_{\eqref{eqn:upgradeproblem_cvx}} := \left\{ a \in \{0,1\}^{n_u}
    \; \middle| \;
    \begin{matrix} Aa \le b, \\
      \exists (Z,y) : \eqref{eqn:up3_lmi}, \eqref{eqn:up3_psd} \\
      \forall i \in \{1,\ldots,I\}
    \end{matrix} 
  \right\}.
\end{equation}
We now state a theorem that relates policy cuts and
problems~\eqref{eqn:upgradeproblem2} and~\eqref{eqn:upgradeproblem_cvx}:
\begin{theorem}
  \label{thm:globopt}
  Let \emph{(R$^+$)} be problem~\eqref{eqn:upgradeproblem_cvx} with additional
  constraints~\eqref{eqn:lazycut} added for each $a^r \in \mathcal A$, where
  \begin{equation*}
    \mathcal A := \left\{ a \in \{0,1\}^{n_u} \;\middle|\;
      a \in \mathcal F_{\eqref{eqn:upgradeproblem_cvx}} \setminus
        \mathcal F_{\eqref{eqn:upgradeproblem2}} \right\}.
  \end{equation*}
  Let $\mathcal P$ be the intersection of these constraints, which is a 
  finite-dimensional convex polyhedron:
  \[ \mathcal P := \left\{ a \in \R^{n_u} \; \middle| \; \|a - a^r\|_1 \ge 1
    \;\forall a^r \in \mathcal A \right\}. \]
  This allows the projection of the feasible region of (R$^+$) onto the upgrade
  space to be written as
  \[ \mathcal F_{\text{\emph{(R$^+$)}}} := \mathcal 
    F_{\eqref{eqn:upgradeproblem_cvx}} \cap \mathcal P. \]
  Problem \emph{(R$^+$)} is now equivalent to~\eqref{eqn:upgradeproblem2} in the
  following sense: $a \in \mathcal F_{\eqref{eqn:upgradeproblem2}}$ if and only
  if $a \in \mathcal F_{\text{\emph{(R$^+$)}}}$. Additionally, \emph{(R$^+$)}
  and~\eqref{eqn:upgradeproblem2} have the same sets of globally optimal
  choices of $a$.

\end{theorem}
\begin{proof}
  First note that $\mathcal F_{\eqref{eqn:upgradeproblem2}} \subseteq 
   \mathcal F_{\eqref{eqn:upgradeproblem_cvx}}$: An optimizer
  $(a^*,z^*,y^*)$ of~\eqref{eqn:upgradeproblem2} can be used to compute a $Z^*
  = (z^*)(z^*)^T$. This $Z^*$ is positive semidefinite and hence
  $(a^*,Z^*,y^*)$ is feasible for~\eqref{eqn:upgradeproblem_cvx}. 
  However, any $a$ that are in $\mathcal F_{\eqref{eqn:upgradeproblem_cvx}} \setminus
  \mathcal F_{\eqref{eqn:upgradeproblem2}}$ are excluded from (R$^+$)
  by the additional constraints in $\mathcal P$, in other words:
  \[ \mathcal F_{\text{(R$^+$)}} = \mathcal F_{\eqref{eqn:upgradeproblem_cvx}} \setminus
    (\mathcal F_{\eqref{eqn:upgradeproblem_cvx}} \setminus 
    \mathcal F_{\eqref{eqn:upgradeproblem2}}) = 
      \mathcal F_{\eqref{eqn:upgradeproblem2}}.
  \]
  Because the two problems have the same feasible sets of $a$ and the cost
  functions are the same and only depend on $a$, the two problems are
  equivalent. 
\end{proof}

\subsection{Algorithm}
\begin{figure}
  \vspace{0.2cm}
  \begin{algorithmic}[1]
    \small
    \State Set $U = \infty, L = -\infty$, 
      Tree: Root vertex $\mathcal I_0 = \mathcal I_1 = \emptyset$ 
    \label{algstep:termination}
    \While{$U - L > \varepsilon$}
      \State Pick an unprocessed vertex $\mathcal N$ with index
      sets $\mathcal I_0^{\mathcal N}$, $\mathcal I_1^{\mathcal N}$
      \State Solve~\eqref{eqn:upgradeproblem_cvx} with~\eqref{eqn:up3_bin} replaced by 
      \[ a_i \in \begin{cases} \{0\}, & \text{if } i \in \mathcal I_0^{\mathcal N}, \\
        \{1\}, & \text{if } i \in \mathcal I_1^{\mathcal N}, \\
          [0,1] & \text{otherwise.} \end{cases} 
      \] \label{alg:solveref}
      
      \If{Problem was feasible}
        \State Let $(a^{\mathcal N}, s^{\mathcal N}, v^{\mathcal N})$ refer to the solution
        \If{$f(a^{\mathcal N}) < U$ and $a^{\mathcal N} \in \{0,1\}^{n_u}$} 
          \vspace{0.1cm}
          \State Evaluate policy $g(a^{\mathcal N}, s^{\mathcal N}_{\mathrm{load}})$ 
            \label{alg:policycheck}
          \vspace{0.1cm}
          \If{Feasible}
            \vspace{0.1cm}
            \State Update $U = f(a^{\mathcal N})$ \label{alg:updateu}
            \vspace{0.1cm}
          \Else
            \vspace{0.1cm}
            \State Add cut $\|a - a^{\mathcal N}\|_1 \ge 1$
            \State Go back to the solve step (line \ref{alg:solveref})
            \vspace{0.1cm}
          \EndIf \label{alg:policyend}
          \vspace{0.1cm}
        \ElsIf{$f(a^{\mathcal N}) < U$ but $a^{\mathcal N} \not\in \{0,1\}^{n_u}$}
          \vspace{0.1cm}
          \State Select index $k$, $k \not\in \mathcal I_0^{\mathcal N} 
            \cup \mathcal I_1^{\mathcal N}$
          \vspace{0.1cm}
          \State Add a vertex with $\mathcal I_0 = \mathcal I_0^{\mathcal N} 
            \cup \{k\}, \mathcal I_1 =\mathcal  I_1^{\mathcal N}$
          \vspace{0.1cm}
          \State Add a vertex with $\mathcal I_0 = \mathcal I_0^{\mathcal N} ,
            \mathcal I_1 = \mathcal I_1^{\mathcal N}\cup \{k\}$
          \vspace{0.1cm}
        \EndIf 
      \Else\vspace{0.1cm}
        \State Set $f(a^{\mathcal N}) = \infty$
        \vspace{0.1cm}
      \EndIf \vspace{0.1cm}
      \State Update $L = \min \left\{ L^{\mathcal N}\; |\; \mathcal N \in \text{tree} \right\}$ 
      \vspace{0.1cm}
    \EndWhile
  \end{algorithmic}
  \caption{Branch-and-Bound algorithm with policy cuts. The difference between
    regular Branch-and-Bound and the algorithm here is the policy evaluation
    and cut addition. In regular Branch-and-Bound, lines~\ref{alg:policycheck}
    through~\ref{alg:policyend} would be replaced by only~\ref{alg:updateu}. }
  \label{alg:bnbcuts}
\end{figure}
Theorem~\ref{thm:globopt} does not yield any algorithmic insight into how the
original problem can be solved: The set $\mathcal A$ is not known, and the only
way to compute it is to evaluate the policy function for each $a \in
\{0,1\}^{n_u}$ that satisfies $Aa \le b$. 

However, the field of mixed-integer convex programming provides the idea of
\emph{lazy constraints}: If a large set of constraints are required in a
problem but only a few of them are expected to be violated, they are not added to the
formulation at the beginning. Rather, the problem is first solved without them.
As soon as a solution is found, a check is performed whether any of the lazy
constraints are violated. If no constraints are violated, the optimal solution
to the true problem was found. If constraint violations are found, the solution
is discarded and the problem resolved with the violated constraints added to
the formulation. A similar approach is taken in other fields where a set of
constraints is too large to enumerate or otherwise
impractical~\cite{Donovan2006}.
In practice, such mixed-integer convex problems are not solved to optimality
before the lazy constraints are checked. Rather, any potential incumbent
solution discovered while traversing the Branch-and-Bound tree is checked for
the lazy constraints and only accepted if it does not violate any of them.
This is also the approach taken in this work. 

The complete algorithm is outlined in Figure~\ref{alg:bnbcuts}. We denote
$\mathcal I_0^\mathcal N$ and $\mathcal I_1^\mathcal N$ to be the index sets of
upgrade variables that are fixed to $0$ and $1$ at tree vertex $\mathcal N$,
respectively. Any feasible solution of~\eqref{eqn:upgradeproblem_cvx} that is
found is checked for policy feasibility and only accepted if it is also
feasible for~\eqref{eqn:upgradeproblem2}.

\section{Numerical simulation study}
\label{sec:numsim}
The algorithm from Figure~\ref{alg:bnbcuts} was implemented in the Julia
language~\cite{Bezanson2017} using the JuMP optimization modeling
package~\cite{DunningHuchetteLubin2017}. In order to accelerate it, a simple
greedy heuristic was run to obtain feasible integer solutions more quickly. The
smooth nonlinear problems encountered in the heuristic were solved using
IPOPT~\cite{Waechter2006} and the semidefinite relaxations were solved with
MOSEK~\cite{ApS2015}.  The clique decomposition
method~\cite{jabr2012exploiting} was applied to improve scalability of the
method to larger power systems. The computer used to run the experiments had an
Intel Core i7-4600U CPU clocked at 2.10 GHz and 8 GB of memory. The operating
system used was 64-bit Debian Linux 9.0, kernel version 4.12.0.

\subsection{Operating policies}
\label{ssec:operatingpolicies}
In the following experiments, the planning problem was solved for different
operating policies $g(a,s)$:
\begin{enumerate}[(i)]
  \item No policy constraints: Solve~\eqref{eqn:upgradeproblem_cvx} to global
    optimality, neglecting any policy constraints and relaxation gaps.
    This can be seen as the reference approach.
  \item AC optimal power flow policy: An AC OPF problem is solved, starting at
    the voltage vector with all entries equal to $1.0$ per unit, using IPOPT as
    solver. Slack variables are added to the voltage magnitude and line current
    constraints in order to have the policy return a result even if it could
    only find one that violates the constraints. The objective is set to
    minimize active and reactive power usage at the generators.

  \item Newton AC power flow: A newton AC power flow is started with the  
    snapshot voltages. 
\end{enumerate}
All of these policies can be evaluated for a given upgrade choice $a$ and a
given load profile $s_\mathrm{load}$. The resulting system state is then checked
for feasibility with respect to the operating constraints. 

\begin{figure}[t]
  \vspace{0.5cm}
  \centering
  \includegraphics[width=0.3\textwidth]{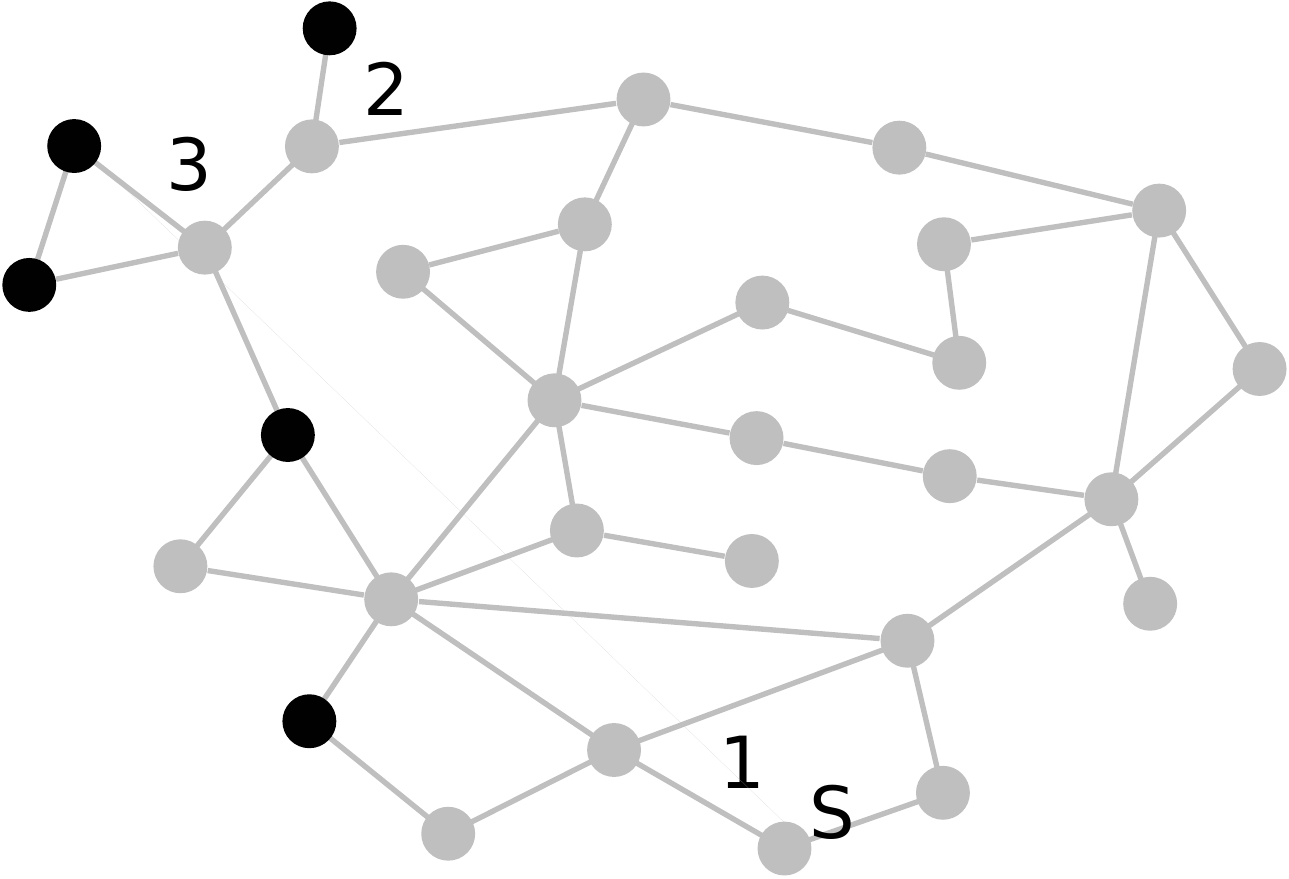}
  \caption{Violations and upgrades for the IEEE 30-bus test case with tightened
    voltage limits. Darker vertices represent voltage magnitude violations. The
    vertex marked "S" is the slack bus. The numbers are used as references in
    the upgrade summary table.}
  \label{fig:30bus}
\end{figure}

\begin{figure}[t]
  \vspace{0.25cm}
  \centering
  \includegraphics[width=0.46\textwidth]{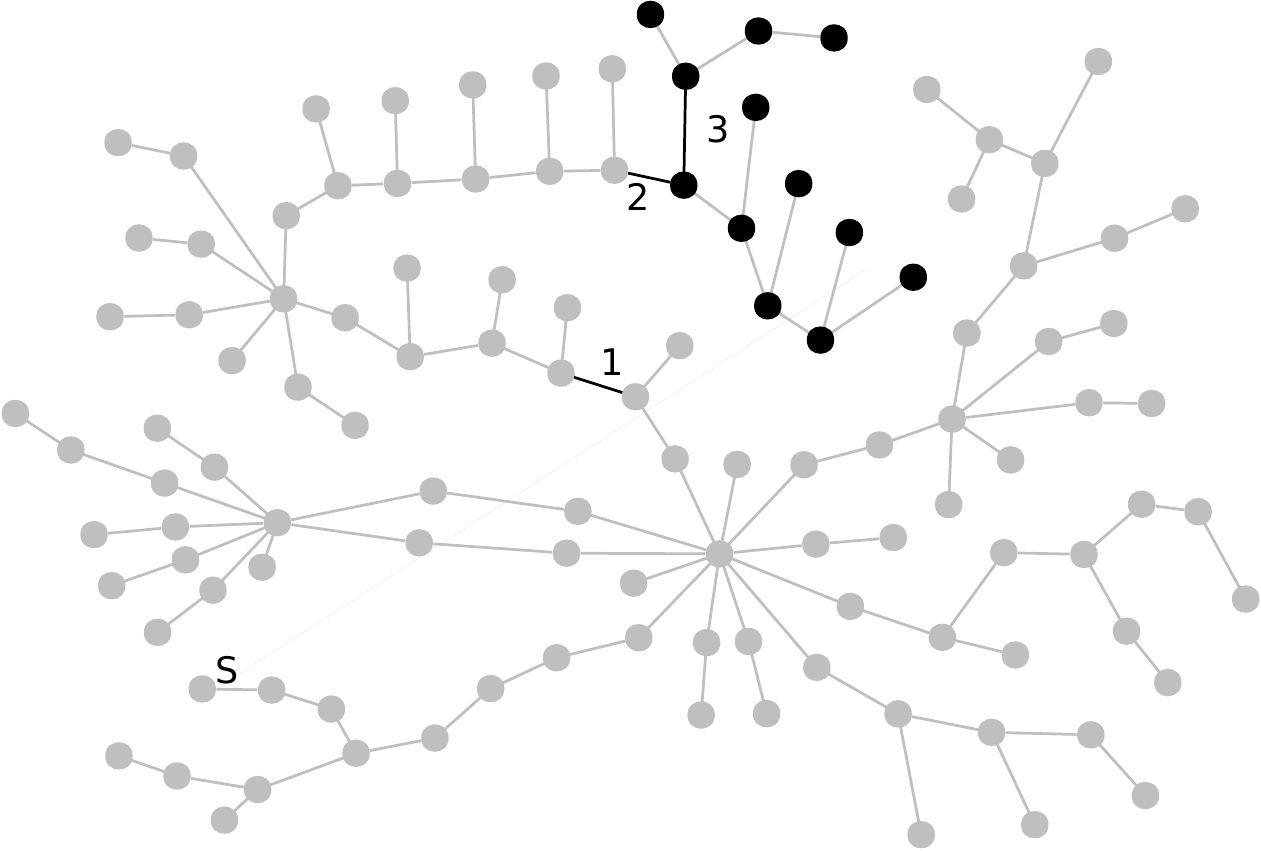}
  \caption{Violations and upgrades for a part of the Z\"urich distribution grid
  with realistic upgrade data. Darker lines represent current violations. }
  \label{fig:116bus}
\end{figure}

\begin{table}[t]
  \centering
  \vspace{0.1cm}
  \renewcommand{\arraystretch}{1.3}
  \caption{Upgrade results for 30-bus case}
  \begin{tabular}{l c c c}
    \hline
    & \textbf{No policy} & \textbf{OPF} & \textbf{Newton PF}\\
    Upgrades            & None & None & 1,2,3 \\
    Cost                & 0 & 0 & 3 \\
    Slack active power  & No result   & -35.8\,MW & 150\,MW \\
    Slack reactive power  & No result & 26.7\,MVar & 1.2\,MVar \\
    Average $|v|$ slack & No result & 0.020 p.u. & 0.014 p.u. \\
    \hline
  \end{tabular}
  \label{tbl:comparison_30}
\end{table}

\begin{table}[t]
  \centering
  \vspace{0.3cm}
  \renewcommand{\arraystretch}{1.3}
  \caption{Upgrade results for 116-bus case}
  \begin{tabular}{l c c c}
    \hline
    & \textbf{No policy} & \textbf{OPF} & \textbf{Newton PF}\\
    Upgrades            & None & 3 & 1,2,3 \\
    Cost                & 0 & 1 & 3 \\
    Slack active power  & No result   & 3.19\,MW    & 3.15\,MW    \\
    Slack reactive power  & No result & 0.63\,MVAr & 0.64\,MVAr \\
    Average $|v|$ slack & No result & 0.054 p.u. & 0.050 p.u.\\
    \hline
  \end{tabular}
  \label{tbl:comparison_116}
\end{table}

\subsection{Case study: IEEE 30-bus network}
As a first demonstration, consider the IEEE 30-bus system from
MATPOWER~\cite{Zimmerman2011}. Being far from radial, the guarantees
for convexification are not guaranteed to hold. We use the existing load data
from MATPOWER along with tightened voltage limits $[1.01,1.07]$ as a violating
snapshot. For upgrades, each line was allowed to be upgraded to $1.5$ or $3$
times its original admittance, at constant ratio of conductance to susceptance. 

Figure~\ref{fig:30bus} shows the network and locations of violations and
Table~\ref{tbl:comparison_30} shows the upgrades obtained for different policies.
When the optimization was run, both the relaxed
problem~\eqref{eqn:upgradeproblem_cvx} as well as the problem with the AC~OPF
policy constraint (ii) required no upgrades. Re-dispatch from the AC~OPF was
sufficient to satisfy the tightened voltage constraints. Note that
the re-dispatch of the AC~OPF led to active power flowing out through the slack
bus in the system -- other generators were dispatched to generate more power in
this case.  For the Newton AC power flow policy (iii), upgrades were required:
After an hour of optimization, the lower bound was raised to 2 upgrades,
whereas the best solution found had 3 upgrades.

\subsection{Case study: 116-bus network}
This example studies part of the distribution grid in Z\"urich, Switzerland.
This grid has 116 buses and is almost radial: One cycle of length 6 is present.
Actual load data was used, enhanced with some fictional generation to simulate
a large curtailable PV installation far away from the slack bus. The resulting
data leads to a steady state that violates both the voltage bounds and some
current bounds. Physically available line parameters were used for upgrade
possibilities, for a total of 525 possible upgrades (locations and strengths
were picked in a case study and discussion with the DSO).

Figure~\ref{fig:116bus} shows this network and locations of violations and
Table~\ref{tbl:comparison_116} shows the upgrades obtained.
In this application, the solution of the problem without any operating policy is
to not perform any upgrades. This is due to the SDP
relaxation~\eqref{eqn:upgradeproblem_cvx} being feasible for the operational
constraints, despite OPF solvers not finding any feasible solutions. In this
case, a system designer would not get much information from this result, since
not upgrading the system is not an option. When the AC~OPF policy constraint
(i.e., policy (ii)) is added, a solution with one upgrade is found. This
solution has the added benefit of a certificate that if the AC~OPF operating
policy is used in practice, the upgrades as found by the procedure are certified
to work for the given load pattern. If policy (iii) is used, two additional
upgrades are required for the operation since the policy does not optimize the
power dispatch. The optimization with policy (ii) found a first solution with 4
upgrades within seconds, then took 10 minutes to improve that solution to one
upgrade and close the gap. For policy (iii), a solution with 4 upgrades was also
found quickly, but after an hour of runtime, the best solution found was 3
upgrades with a lower bound of 2.

\subsection{Discussion}
The addition of policy constraints opens many possibilities for better modeling
of the real system and its operation. However, since there is no analytical
information about the policy function $g(a,s)$, each policy cut inherently only
removes one upgrade configuration from consideration. In other words, if a lot
of policy cuts are required, the performance of the method approaches that of a
brute force evaluation of all upgrade possibilities.

In the experiments, it was found that the AC OPF operating policy is rather
good at finding feasible points if they exist --- meaning the set $\mathcal
F_{\eqref{eqn:upgradeproblem_cvx}} \setminus \mathcal
F_{\eqref{eqn:upgradeproblem2}}$ was found to be small.  However, the Newton
power flow policy fared worse in this respect: Many policy cuts were required.
It would therefore be advisable to include analytical information such as
convex relaxations of $g(a,s)$ into Problem~\eqref{eqn:upgradeproblem_cvx}, if
available.

The policy constraints can also be used to investigate trade-offs between
upgrades to the system and policy changes. For example, it was found that if
the PV installation in the 116-bus example is allowed to be further curtailed,
fewer upgrades will be required. The base case in the experiment was a 20\
curtailment. If no curtailment is allowed, the same upgrades as for the Newton 
power flow are required since the only dispatchable power is the slack bus. If
curtailment is raised to 25\
scenario.

\section{Conclusion}
The method presented in this work outlines a systematic way of using convex
relaxations and heuristic methods, as well as operating policy constraints to
find globally optimal solutions to practical power system line upgrade problems.
The addition of the operating policy constraints both made the problem more
realistic as well as circumventing the hardness of checking AC power flow
feasibility. An industrial example with real world data was presented,
illustrating the features of the problem and demonstrating the effectiveness of
the approach.

\section*{Acknowledgments}
\label{sec:acknowledgements}
This work was supported by the Swiss Commission for Technology and Innovation
(CTI), (Grant 16946.1 PFIW-IW). We thank the team at Adaptricity (Stephan Koch,
Andreas Ulbig, Francesco Ferrucci) for providing the system data and valuable
discussions on power systems.

\bibliographystyle{ieeetr}
\bibliography{biblio}

\begin{thebibliography}{10}

\bibitem{Walling2008}
R.~Walling, R.~Saint, R.~Dugan, J.~Burke, and L.~Kojovic, ``Summary of
  distributed resources impact on power delivery systems,'' {\em IEEE Trans. on
  Power Delivery}, vol.~23, pp.~1636--1644, July 2008.

\bibitem{Merkli2017}
S.~Merkli, A.~Domahidi, J.~Jerez, M.~Morari, and R.~S. Smith, ``Fast {AC}
  {P}ower {F}low {O}ptimization using {D}ifference of {C}onvex {F}unctions
  {P}rogramming,'' {\em IEEE Trans. on Power Systems}, vol.~33, pp.~363--372,
  2018.

\bibitem{Rider2007}
M.~J. Rider, A.~V. Garcia, and R.~Romero, ``{Power system transmission network
  expansion planning using AC model},'' {\em IET Generation, Transmission
  Distribution}, vol.~1, pp.~731--742, September 2007.

\bibitem{Ramirez-Rosado1998}
I.~J. Ramirez-Rosado and J.~L. Bernal-Agustin, ``Genetic algorithms applied to
  the design of large power distribution systems,'' {\em IEEE Trans. on Power
  Systems}, vol.~13, pp.~696--703, May 1998.

\bibitem{Macedo2016}
L.~H. Macedo, C.~V. Montes, J.~F. Franco, M.~J. Rider, and R.~Romero, ``{An
  MILP Branch Flow Model for Concurrent AC Multistage Transmission Expansion
  and Reactive Power Planning with Security Constraints},'' {\em IET
  Generation, Transmission \& Distribution}, vol.~10, pp.~3023--3032, 2016.

\bibitem{Jabr2013}
R.~A. Jabr, ``{Optimization of AC transmission system planning},'' {\em IEEE
  Trans. on Power Systems}, vol.~28, no.~3, pp.~2779--2787, 2013.

\bibitem{Bienstock2015}
D.~Bienstock and A.~Verma, ``{Strong NP-hardness of AC power flows
  feasibility},'' {\em arXiv:1512.07315}, 2015.

\bibitem{Lehmann2016}
K.~Lehmann, A.~Grastien, and P.~V. Hentenryck, ``{AC-Feasibility on Tree
  Networks is NP-Hard},'' {\em IEEE Trans. on Power Systems}, vol.~31,
  pp.~798--801, Jan 2016.

\bibitem{Low2014b}
S.~H. Low, ``{Convex Relaxation of Optimal Power Flow - Part I: Formulations
  and Equivalence},'' {\em IEEE Trans. on Control of Network Systems}, vol.~1,
  pp.~15--27, March 2014.

\bibitem{Merkli2017a}
S.~Merkli, ``{A Mixed-Integer QCQP Formulation of AC Power System Upgrade
  Planning Problems}.'' arXiv:1710.09228, 2017.

\bibitem{Low2014a}
S.~H. Low, ``{Convex Relaxation of Optimal Power Flow - Part II: Exactness},''
  {\em IEEE Trans. on Control of Network Systems}, vol.~1, pp.~177--189, June
  2014.

\bibitem{Donovan2006}
S.~Donovan, ``An improved mixed integer programming model for wind farm layout
  optimisation,'' in {\em Proceedings of the 41st annual conference of the
  Operations Research Society}, pp.~143--151, 2006.

\bibitem{Bezanson2017}
J.~Bezanson, A.~Edelman, S.~Karpinski, and V.~B. Shah, ``Julia: A fresh
  approach to numerical computing,'' {\em SIAM Review}, vol.~59, no.~1,
  pp.~65--98, 2017.

\bibitem{DunningHuchetteLubin2017}
I.~Dunning, J.~Huchette, and M.~Lubin, ``{JuMP: A Modeling Language for
  Mathematical Optimization},'' {\em SIAM Review}, vol.~59, no.~2,
  pp.~295--320, 2017.

\bibitem{Waechter2006}
A.~W{\"a}chter and L.~T. Biegler, ``On the implementation of an interior-point
  filter line-search algorithm for large-scale nonlinear programming,'' {\em
  Math. prog.}, vol.~106, no.~1, pp.~25--57, 2006.

\bibitem{ApS2015}
ApS, {\em The MOSEK C optimizer API manual Version 7.0}, 2015.

\bibitem{jabr2012exploiting}
R.~A. Jabr, ``{Exploiting sparsity in SDP relaxations of the OPF problem},''
  {\em IEEE Trans. on Power Systems}, vol.~27, no.~2, pp.~1138--1139, 2012.

\bibitem{Zimmerman2011}
R.~Zimmerman, C.~Murillo-S{\'a}nchez, and R.~Thomas, ``{MATPOWER: Steady-State
  Operations, Planning, and Analysis Tools for Power Systems Research and
  Education},'' {\em IEEE Trans. on Power Systems}, vol.~26, pp.~12--19, Feb
  2011.

\end{thebibliography}

\end{document}